\documentclass[12pt]{article}
\usepackage{amsmath,amssymb,CJK}\usepackage{mathrsfs}
\usepackage{amsfonts}
\usepackage{amsmath,amssymb}

\usepackage{amsthm}
\usepackage{amstext}
\usepackage{amssymb}
\usepackage{mathrsfs}
\usepackage{amscd}
\usepackage{xypic}
\usepackage{epsf}              
\usepackage{graphicx}          
\usepackage{fancybox}          
\usepackage{color}             
\usepackage{fancyhdr}
\usepackage[hang,footnotesize]{caption2}

\baselineskip 5.5pt \lineskip 5.5pt \numberwithin{equation}{section}

\def\qed{\hfill$\Box$\par}

\def \R{\hbox{$I\hskip -3pt R$}}

\def\qed{\ \ \ifhmode\unskip\nobreak\fi\ifmmode\ifinner
         \else\hskip5pt\fi\fi
 \hbox{\hskip5pt\vrule width4pt height6pt depth1.5pt\hskip 1 pt}}

\def\cl{\centerline}

\def\vs{\vspace*}

\def\R{\mathbb{R}}

\def\K{\mathbb{K}}

\def\mathcal{\mathscr}
\newfont{\aaa}{cmb10 at 19pt}
\newfont{\bbb}{cmb10 at 11pt}

\newtheorem{theo}{Theorem}[section]
\newtheorem{lemm}[theo]{Lemma}
\newtheorem{rema}[theo]{Remark}
\newtheorem{defi}[theo]{Definition}

\newtheorem{prop}[theo]{Proposition}
\newtheorem{exam}[theo]{Example}
\begin{document}

\cl {{\Large\bf Representations of Bihom-Lie algebras}
\noindent\footnote{Supported by the National Science Foundation of
China (No. 11047030) and the Science and Technology
Program of Henan Province (No. 152300410061).}} \vs{6pt}

\cl{ Yongsheng Cheng$^{1), 2)}$, Huange Qi$^{1)}$}

\cl{$1)$ \small School
of Mathematics and Statistics, Henan
University, Kaifeng 475004, China}
\cl{ $2)$ \small Institute of Contemporary Mathematics, Henan
University, Kaifeng 475004, China}

\vs{15pt}
\normalsize\noindent{\bbb Abstract}\quad
Bihom-Lie algebra is a generalized Hom-Lie algebra endowed with two
 commuting multiplicative linear maps.
 In this paper, we study cohomology and representations of Bihom-Lie
 algebras. In particular, derivations,
central extensions, derivation extensions, the trivial representation
and the adjoint representation of Bihom-Lie algebras
are studied in detail.\vspace{0.3cm}

\noindent{\bbb Keywords}\quad
Bihom-Lie algebras; Derivations; Cohomology; Representations of Bihom-Lie algebras
\\[4mm]

\section{Introduction}

The motivations to study Hom-Lie structures are related to physics
and to deformations of Lie algebras, in particular Lie algebras of
vector fields. Hom-Lie algebra, introduced by Hartig, Larson and Silvestrov in \cite{HLS},
is a triple $(L, [\cdot, \cdot],\alpha_{L}) $
 consisting of a vector space $L$, a bilinear map
 $[\cdot, \cdot]: L\otimes L\rightarrow L$  and a
 vector space homomorphism $\alpha_{L}: L\rightarrow L $
 satisfying the following conditions
$$[x,y]=-[y,x], ~~~(skew-symmetry),$$
$$[\alpha(x),[y,z]]+[\alpha(y),[z,x]]+[\alpha(z),[x,y]]=0,~~~(Hom-Jacobi~~ identity)$$
for all  $x,y,z\in L.$
The main feature of these algebras
is that the identities defining the structures are twisted by homomorphisms.
The paradigmatic examples are $q$-deformations of
Witt and Virasoro algebras, Heisenberg-Virasoro algebra and other algebraic structure
constructed in pioneering works
\cite{ CS2, CY, HLS, LCM, MS1, SX, Y1}.
The representation theory of Hom-Lie algebras
was introduced by Sheng in \cite{S}, in which, Hom-cochain complexes, derivation,
central extension, derivation extension, trivial representation and adjoint representation
of Hom-Lie algebras were studied. In addition,
(co)homology and deformations theory of Hom-algebras were studied in
\cite{AEM, CS1, HSS, MS2, SX, Y2}.

In \cite{GMMP}, the authors introduced a generalized algebraic structure endowed with two
commuting multiplicative linear maps, called Bihom-algebras. When the two linear maps are same, then
Bihom-algebras will be return to Hom-algebras.
These algebraic structures include Bihom-associative algebras, Bihom-Lie algebras and
Bihom-bialgebras. The purpose of this paper is to study cohomology,
representations of Bihom-Lie algebras. In particular, the trivial representation
and the adjoint representation of Bihom-Lie algebras
are studied in detail. Derivations, central extensions and derivation extensions of Bihom-Lie
algebras are also studied as an application.

The paper is organized as follows. In Section 2 we give
the definition, properties of Bihom-Lie algebras and a few examples of Bihom-Lie algebra.
In Section 3 we study derivations, inner derivation and
derivation extension of multiplicative Bihom-Lie algebras.
In Section 4 we give representations of regular Bihom-Lie
algebras, and give their cohomology.
 In Section 5 we study the trivial representations of regular
 Bihom-Lie algebras, and show that central extensions of a
 Bihom-Lie algebra are controlled by the second cohomology
 in the trivial representation. In Section 6 we study the adjoint representation of a regular
Bihom-Lie algebra. We obtain that a $1$-cocycle
associated to the adjoint representation is a derivation for Bihom-Lie algebras.

\section{Bihom-Lie algebras}

Our aim in this section is to introduce more general
definition of Bihom-Lie algebras and properties  from \cite{GMMP}.
From now on, we will always work over  a base field $\K$.
All algebras, linear spaces etc. will be over $\K$.

\begin{defi}
\end{defi}
\flushleft
\begin{enumerate}
\item A \textit{BiHom-associative algebra} over $\Bbbk $
is a 4-tuple $( A,\mu ,\alpha ,\beta )$, where $A$ is
a $\Bbbk $-linear space, $\alpha \colon A\rightarrow A$, $\beta \colon A\rightarrow A$
and $\mu \colon A\otimes A\rightarrow A$ are linear maps, with notation $\mu
 ( a\otimes a^{\prime } ) =aa^{\prime }$, satisfying the following
conditions, for all $a,a^{\prime },a^{\prime \prime }\in A$:
\begin{gather*}
\alpha \circ \beta =\beta \circ \alpha , \\
\alpha (a)  ( a^{\prime }a^{\prime \prime } ) = (
aa^{\prime } ) \beta  ( a^{\prime \prime } )  \quad \text{(BiHom-associativity)}.  
\end{gather*}
~~~~~~In particular, if $\alpha  ( aa^{\prime } ) =\alpha (a) \alpha  (
a^{\prime } )$ and $\beta  ( aa^{\prime } ) =\beta
(a) \beta  ( a^{\prime } )$, we call it multiplicative Bihom-associative algebra.
\item  A  \textit{Bihom-Lie algebra} over a field $\K $ is a
$4$-tuple $( L,[-] ,\alpha ,\beta ) $, where $L$ is a $\K $-linear
space,  $\alpha \colon L\rightarrow L$, $\beta \colon L\rightarrow L$ and $[-]
\colon L\otimes L\rightarrow L$ are linear maps, with notation $[-]
( a\otimes a^{\prime }) =[ a,a^{\prime }] $,
satisfying the following conditions, for all $a,a^{\prime },a^{\prime \prime
}\in L$:
\begin{gather}
\label{a}\alpha \circ \beta =\beta \circ \alpha , \\
\label{b} [ \beta (a) ,\alpha  ( a^{\prime } )  ] =-[ \beta ( a^{\prime }) ,\alpha (a) ]
\qquad \text{(skew-symmetry)}, \\
\label{c} \big[ \beta ^{2}(a) , [ \beta  ( a^{\prime } )
,\alpha  ( a^{\prime \prime } )  ] \big] +\big[ \beta
^{2} ( a^{\prime } ) , [ \beta  ( a^{\prime \prime } )
,\alpha (a)  ] \big] +\big[ \beta ^{2} ( a^{\prime
\prime } ) , [ \beta (a) ,\alpha  ( a^{\prime
} ) ] \big] =0
\end{gather}
\ \ \ \ \ \ \ \ \ \ \ \ \ \ \ \ \
\ \ \ \ \ \ \ \ \ \ \ \ \ \ \ \ \
\text{(Bihom-Jacobi condition)}.
\item A  Bihom-Lie algebra is called a \textit{multiplicative} Bihom-Lie algebra if $\alpha$ and $\beta$ are algebraic morphisms, i.e. for any
 $ a^{\prime }, a^{\prime \prime }\in L$, we have
\begin{gather}
\alpha ([ a^{\prime },a^{\prime \prime }])=[ \alpha (
a^{\prime }) ,\alpha ( a^{\prime \prime }) ],\quad
  \beta ([ a^{\prime },a^{\prime \prime }])=[
\beta ( a^{\prime }) ,\beta ( a^{\prime \prime }) ].
\end{gather}
\item  A multiplicative Bihom-Lie algebra is called a \textit{regular} Bihom-Lie algebra if $\alpha, \beta$
are bijective maps.
\end{enumerate}


\begin{rema}
\label{rema1}
Obviously, a BiHom-Lie algebra $( L,[-] ,\alpha, \beta ) $ for which
$\alpha=\beta$ is just a Hom-Lie algebra $( L, [-]
,\alpha) $.
\end{rema}

~~~Similar with Lie algebras coming from associative algebras by commutator,
Bihom-Lie algebras can also be obtained by Bihom-associative algebras. Unlike the Hom case,
to obtain a Bihom-Lie algebra from a Bihom-associative algebra we need the maps $\alpha$
and $\beta$ to be bijective. Next proposition comes from \cite{GMMP}.

\begin{prop}
\label{croset} If $ ( A,\mu ,\alpha ,\beta
 ) $ is a Bihom-associative algebra with bijective $\alpha $ and $\beta
$, then, for every $a,a^{\prime }\in A$, we can set
\begin{gather*}
[ a,a^{\prime }] =aa^{\prime }-\big(\alpha ^{-1}\beta  (
a^{\prime } )\big)\big(\alpha \beta ^{-1}(a)\big).
\end{gather*}
Then $ ( A,[-] ,\alpha ,\beta  ) $ is a Bihom-Lie
algebra, denoted by~$L(A)$.
\end{prop}

\begin{exam}
We take an  example of 2-dimensional Bihom-Lie algebra obtained
from $2$-dimensional Bihom-associative algebra . Let $\{e_{1}, e_{2}\}$
be a basis of associative algebra $A$. The linear maps
$\alpha, \beta  \colon A\rightarrow A$ and $\mu
\colon A\otimes A\rightarrow A$ are  defined by
$$\alpha(e_{1})=e_{1},\qquad \alpha(e_{2})=\frac{1}{m}e_{1}
+\frac{n-1}{n}e_{2},\qquad\beta(e_{1})=e_{1}, \qquad \beta(e_{2})=e_{2},$$
$$\mu(e_{1},e_{1})=me_{1}, \qquad \mu(e_{1},e_{2})=me_{2},$$
$$\mu(e_{2},e_{1})=ne_{1}, \qquad \mu(e_{2},e_{2})=ne_{2},$$
where $m, n$ are parameters in $\K$, with $m, n\neq 0$ and $n\neq 1$.
Then we get $(A, \mu, \alpha,\beta)$ is a $2$-dimensional
Bihom-associative algebra. In view of Proposition~\ref{croset}, we have
$$[e_{1}, e_{1}]=0,\qquad [e_{1}, e_{2}]=me_{2}-ne_{1},$$
$$ [e_{2}, e_{1}]=(n-1)e_{1}-\frac{m(n-1)}{n}e_{2},
\qquad [e_{2},e_{2}]=-\frac{n}{m}e_{1}+e_{2}.$$
Then $ ( A,[-] ,\alpha ,\beta  ) $ is a Bihom-Lie
algebra.
\end{exam}

~~~Bihom-Lie algebras can also be induced by Lie
algebras and their homomorphisms \cite{GMMP}.
\begin{prop}
\label{ppp}
Let $ ( L,[-]  ) $ be an ordinary Lie algebra over a
field $\K$ and let $\alpha ,\beta \colon L\rightarrow L$ two commuting linear
maps such that $\alpha  (  [ a,a^{\prime } ]  ) = [
\alpha (a) ,\alpha  ( a^{\prime } )  ]$ and $\beta
 (  [ a,a^{\prime } ]  ) = [ \beta (a)
,\beta  ( a^{\prime } )  ]$, for all $a,a^{\prime }\in L$.
Define the linear map $ \{ - \} \colon L\otimes L\rightarrow L$,
\begin{gather*}
 \{ a,b \} = [ \alpha (a) ,\beta (b) ] ,\qquad \text{for all} \ \ a,b\in L.
\end{gather*}
Then $L_{( \alpha ,\beta ) }:=(L, \{-\}, \alpha ,
\beta )$ is a Bihom-Lie algebra, called the \textit{Yau twist} of $( L,[-]) $.
\end{prop}

\begin{exam}
\label{exam1}
Let $L=sl(2,\K)$, $char\K\neq 2,$ whose standard basis consists of
$$X=\left(
      \begin{array}{cc}
        0 & 1 \\
        0 & 0 \\
      \end{array}
    \right),
Y=\left(
    \begin{array}{cc}
      0 & 0 \\
      1 & 0 \\
    \end{array}
  \right),
  H=\left(
      \begin{array}{cc}
        1 & 0 \\
        0 & -1 \\
      \end{array}
    \right).
$$
Then $[H,X]=2X,\quad [H,Y]=-2Y,\quad [X,Y]=H.$  Two linear maps 
$\alpha, \beta: L\rightarrow L$ are defined by
$$\alpha(X)=X, \quad \alpha(Y)=-k^{2}X+Y+kH,\quad \alpha(H)=-2kX+H,$$
$$\beta(X)=X, \quad \beta(Y)=-l^{2}X+Y+lH,\quad \beta(H)=-2lX+H,$$
where $k,l$ are parameters in $\K$.
Obviously we have $\alpha\circ\beta=\beta\circ\alpha$. 
Then for $X, Y, H\in L,$ we have
$$\alpha[X,Y]=\alpha(H)=-2kX+H,$$
$$\alpha[H,X]=\alpha(2X)=2X,$$
$$\alpha[H,Y]=\alpha(-2Y)=2k^{2}X-2Y-2kH.$$
On the other hand, we have
 $$[\alpha(X), \alpha(Y)]=[X, -k^{2}X+Y+kH]=H-2kX,$$
 $$[\alpha(H),\alpha(X)]=[-2kX+H,X]=2X,$$
 $$[\alpha(H),\alpha(Y)]=[-2kX+H,-k^{2}X+Y+kH]=-2kH+2k^{2}X-2Y.$$
Therefore, for $a, a^{\prime}\in L,$  we have 
$$\alpha[a,a^{\prime}]=[\alpha(a), \alpha(a^{\prime})].$$
 Similarly, we have $$\beta[a,a^{\prime}]=[\beta(a), \beta(a^{\prime})].$$
 In view of Proposition~\ref{ppp}, we define the 
 linear map $ \{ - \} \colon L\otimes L\rightarrow L,$
 \begin{gather*}
 \{ a,b \} = [ \alpha (a) ,\beta (b) ] ,\qquad \text{for all} \ \ a,b\in L.
\end{gather*}
Then $L_{( \alpha ,\beta ) }:=(L, \{-\}, \alpha ,
\beta )$ is a Bihom-Lie algebra. Namely, $\big(sl(2,\K), \{-\}, \alpha ,
\beta \big)$ is a Bihom-Lie algebra.
 \end{exam}

\begin{rema}
We present an example of a Bihom-Lie algebra that 
cannot be expressed as a Hom-Lie algebra. By Example~\ref{exam1},
$\big(L=sl(2,\K), \{-\}, \alpha ,\beta \big)$ is a Bihom-Lie algebra. If we set
$$\alpha(X)=X, \quad \alpha(Y)=-k^{2}X+Y+kH,\quad \alpha(H)=-2kX+H,$$
$$\beta(X)=X, \quad \beta(Y)=Y,\quad \beta(H)=H,$$
where $k$ are parameters in $\K$, with $k\neq 0$.  
For every $a,b\in L$ we have  $$\{-\}(a\otimes b)
=[\alpha(a), \beta(b)]=[\alpha(a), b].$$  Now we 
check that $(L, \{-\}, \alpha )$ is not a Hom-Lie algebra.  We only need to verify:
\begin{eqnarray*}
 & &\big\{\alpha(X),\{Y,H\} \big\}+\big\{\alpha(Y),\{H,X\} 
 \big\}+\big\{\alpha(H),\{X,Y\} \big\}\\
 &=&\big\{\alpha(X),[\alpha(Y),H] \big\}+\big\{\alpha(Y),
 [\alpha(H),X] \big\}+\big\{\alpha(H),[\alpha(X),Y] \big\}\\
 &=&\big\{X, [-k^{2}X+Y+kH,H] \big\}
 +\big\{ -k^{2}X+Y+kH,[-2kX+H,X] \big\}+\big\{-2kX+H,[X ,Y] \big\}\\
 &=&\big\{X, 2k^{2}X+2Y \big\}
 +\big\{ -k^{2}X+Y+kH,2X \big\}+\big\{-2kX+H,H\big\}\\
 &=&[\alpha(X), 2k^{2}X+2Y ]+[ \alpha(-k^{2}X+Y+kH),2X ]+[\alpha(-2kX+H),H]\\
 &=&[X,2k^{2}X+2Y ]+[-4k^{2}X+Y+2kH, 2X]+[-4kX+H,H]\\
 &=&2H+(-2H+8kX)+8kX\\
 &=&16kX\neq0 ~~(since ~ k\neq0).
\end{eqnarray*}
Thus, $(L, \{-\}, \alpha )$ is not a Hom-Lie algebra 
even though $\big(L=sl(2,\K), \{-\}, \alpha, \beta \big)$ is a Bihom-Lie algebra.

\end{rema}

~~~ Consider the direct sum of two Bihom-Lie algebras, we have
\begin{prop}
Given two Bihom-Lie algebras $( L, [-] ,\alpha, 
\beta )$ and $(L^{\prime}, [-]^{\prime} ,\alpha^{\prime}, 
\beta^{\prime})$, there is a Bihom-Lie
algebra $(L\oplus L^{\prime}, [\cdot, \cdot]_{L\oplus L^{\prime}}, 
\alpha+\alpha^{\prime}, \beta+\beta^{\prime} ),$ 
where the skew-symmetric bilinear map
 $[\cdot, \cdot]_{L\oplus L^{\prime}} : 
 \wedge^{2}(L\oplus L^{\prime})\rightarrow L\oplus L^{\prime}$  is given by
$$[(u_{1}, v_{1}), (u_{2}, v_{2})]_{L\oplus L^{\prime}}
= ([u_{1}, u_{2}], [v_{1}, v_{2}]^{\prime}), 
\forall u_{1}, u_{2}\in L, v_{1}, v_{2} \in L^{\prime},$$
and the linear maps $(\alpha+\alpha^{\prime}), 
(\beta+\beta^{\prime}):L\oplus L^{\prime}
\rightarrow L\oplus L^{\prime}$ are given by
$$(\alpha+\alpha^{\prime})(u,v)=\big(\alpha(u), 
\alpha^{\prime}(v)\big),$$
$$(\beta+\beta^{\prime})(u,v)=\big(\beta(u), 
\beta^{\prime}(v)\big), \forall u\in L, v\in L^{\prime}.$$
\end{prop}

\begin{proof}
It is easy to verify Eq.(\ref{a})- Eq.(\ref{c})  for
 $(L\oplus L^{\prime}, [\cdot, \cdot]_{L\oplus L^{\prime}},
 \alpha+\alpha^{\prime}, \beta+\beta^{\prime} ).$
\end{proof}

\begin{defi}
A sub-vector space $\mathfrak{H}\subset L$ is a
Bihom-Lie sub-algebra of  $( L, [-] ,\alpha ,\beta )$
if $\alpha(\mathfrak{H})\subset\mathfrak{H}, \beta(\mathfrak{H})
\subset\mathfrak{H}$ and $\mathfrak{H}  $
is closed under the bracket operation $[-]$  , i.e.
$$[u, u^{\prime}]\in\mathfrak{H},  ~\forall ~u, u^{\prime}\in\mathfrak{H}.$$

\end{defi}

~~~ A morphism $f :( L, [-] ,\alpha ,\beta )\rightarrow(L^{\prime},
[-]^{\prime} ,\alpha^{\prime} ,\beta^{\prime})$
 of Bihom-Lie algebras is a linear map $f : L\rightarrow L^{\prime}$
 such that $\alpha^{\prime}\circ f=f\circ\alpha,$ $\beta^{\prime}\circ f=f\circ\beta$
and $f([u,v])=[f(u), f(v)]^{\prime}$,  for all $u, v\in L$.

~~~ Denote by $\phi_{f}\subset L\oplus L^{\prime}$ the graph of
a linear map $f: L\rightarrow L^{\prime}$.

\begin{prop}
 A linear map $f: ( L, [-] ,\alpha ,\beta )\rightarrow(L^{\prime},
 [-]^{\prime} ,\alpha^{\prime} ,\beta^{\prime})$ is a morphism of
  Bihom-Lie algebras if and only if the graph
  $\phi_{f}\subset L\oplus L^{\prime}$ is a Bihom-Lie-sub-algebra
  of $(L\oplus L^{\prime}, [\cdot, \cdot]_{L\oplus L^{\prime}},
  \alpha+\alpha^{\prime}, \beta+\beta^{\prime} ).$
\end{prop}

\begin{proof}
Let $f: ( L, [-] ,\alpha ,\beta )\rightarrow(L^{\prime}, [-]^{\prime},
\alpha^{\prime} ,\beta^{\prime})$ is a morphism of Bihom-Lie algebras,
then for any $u_{1}, u_{2}\in L,$  we have
$$[\big(u_{1}, f(u_{1})\big), \big(u_{2}, f(u_{2})
\big)]_{L\oplus L^{\prime}}=([u_{1},u_{2}], [f(u_{1}),
f(u_{2})]^{\prime})=([u_{1},u_{2}], f[u_{1},u_{2}])$$
Thus the graph $\phi_{f}$ is closed under the bracket
operation $[\cdot,\cdot]_{L\oplus L^{\prime}}$.  Furthermore, we have
$$(\alpha+\alpha^{\prime})\big(u_{1},f(u_{1})\big)=\big(\alpha(u_{1}), \alpha^{\prime}\circ f(u_{1})\big)=\big(\alpha(u_{1}), f\circ\alpha(u_{1})\big)$$
which implies that $$(\alpha+\alpha^{\prime})(\phi_{f})\subset\phi_{f}.$$
Similarly, $$(\beta+\beta^{\prime})(\phi_{f})\subset\phi_{f}.$$
Thus $\phi_{f}$ is a Bihom-Lie-sub-algebra of $(L\oplus L^{\prime},
[\cdot, \cdot]_{L\oplus L^{\prime}}, \alpha+\alpha^{\prime}, \beta+\beta^{\prime} ).$

~~~Conversely, if the graph $\phi_{f}$ is a Bihom-Lie-sub-algebra of
  $(L\oplus L^{\prime}, [\cdot, \cdot]_{L\oplus L^{\prime}},
  \alpha+\alpha^{\prime}, \beta+\beta^{\prime} ),$
  then we have $$[\big(u_{1}, f(u_{1})\big), \big(u_{2},
  f(u_{2})\big)]_{L\oplus L^{\prime}}=\big([u_{1},u_{2}],
  [f(u_{1}), f(u_{2})]^{\prime}\big)\in\phi_{f},$$
 which implies that $$[f(u_{1}), f(u_{2})]^{\prime}=f[u_{1}, u_{2}].$$
 Furthermore, $(\alpha+\alpha^{\prime})(\phi_{f})\subset\phi_{f}$ yields that
  $$(\alpha+\alpha^{\prime})\big(u_{1},f(u_{1})\big)=\big(\alpha(u_{1}), \alpha^{\prime}\circ f(u_{1})\big)\in\phi_{f},$$
  which is equivalent to the condition $\alpha^{\prime}\circ f(u_{1})=f\circ\alpha(u_{1}) ,$ i.e. $\alpha^{\prime}\circ f=f\circ\alpha.$
  Similarly, $\beta^{\prime}\circ f=f\circ\beta.$
  Therefore, $f$ is a morphism of Bihom-Lie algebras.
\end{proof}

\section{Derivations of Bihom-Lie algebras}

 In this section, we will study derivations of  Bihom-Lie algebras. Let $( L,[-] ,\alpha ,\beta ) $ be a regular Bihom-Lie algebra.  For any integer $k, l$, denote by $\alpha^{k}$ the $k$-times composition of $\alpha$ and $\beta^{l}$ the $l$-times composition of $\beta$, i.e.
 $$\alpha^{k}=\alpha\circ\cdots\circ\alpha (k-times), \beta^{l}=\beta\circ\cdots\circ\beta(l-times).$$
 Since the maps $\alpha, \beta$ commute, we denote by $$\alpha^{k}\beta^{l}=\underbrace{\alpha\circ\cdots\circ\alpha}_{k-times}\circ\underbrace{\beta\circ\cdots\circ\beta}_{l-times}.$$
In particular, $\alpha^{0}\beta^{0}=Id, \alpha^{1}\beta^{1}=\alpha\beta$, $\alpha^{-k}\beta^{-l}$ is the inverse of $\alpha^{k}\beta^{l}.$

\begin{defi}
For any integer $k, l$,  a linear map $D: L\rightarrow L$ is called an $\alpha^{k}\beta^{l}$-derivation of the regular Bihom-Lie algebra $( L,[-] ,\alpha ,\beta ) $, if
\begin{gather}
\label{e}
D\circ\alpha=\alpha\circ D, \quad D\circ\beta=\beta\circ D,
\end{gather}
and
\begin{gather}
D[u,v]=[D(u), \alpha^{k}\beta^{l}(v)]+[\alpha^{k}\beta^{l}(u), D(v)],\quad \forall u,v\in L.
\end{gather}

\end{defi}

~~~Note first that if $\alpha$ and $\beta$ are bijective, 
the skew-symmetry condition  Eq.(\ref{b}) implies
\begin{gather}
\label{d}[u, v]=-[\alpha^{-1}\beta(v), \alpha\beta^{-1}(u)].
\end{gather}

Denote by $Der_{\alpha^{k}\beta^{l}}(L)$ the set of 
$\alpha^{k}\beta^{l}$-derivations of the Bihom-Lie algebra 
$( L,[-] ,\alpha ,\beta ) $.  For any $u\in L$ satisfying 
$\alpha(u)=u, \beta(u)=u,$  define $D_{k,l}(u)\in gl(L)$ by
$$D_{k,l}(u)(v)=-[\alpha^{k}\beta^{l}(v), u], \quad \forall v\in L.$$
By Eq.(\ref{d}),
\begin{eqnarray*}
D_{k,l}(u)(v)&=&-[\alpha^{k}\beta^{l}(v), u]\\
&=&[\alpha^{-1}\beta(u), \alpha^{k+1}\beta^{l-1}(v)]\\
&=&[u,\alpha^{k+1}\beta^{l-1}(v) ].
\end{eqnarray*}
Then $D_{k,l}(u)$ is an $\alpha^{k+1}\beta^{l}$-derivation, which we call an $\mathit{inner}$ $\alpha^{k+1}\beta^{l}$-derivation. In  fact, we have
$$D_{k,l}(u)(\alpha(v))=-[\alpha^{k+1}\beta^{l}(v),u]=-\alpha[\alpha^{k}\beta^{l}(v),u]=\alpha \circ D_{k,l}(u)(v),$$

$$D_{k,l}(u)(\beta(v))-[\alpha^{k}\beta^{l+1}(v),u]=-\beta[\alpha^{k}\beta^{l}(v),u]=\beta\circ D_{k,l}(u)(v),$$
which implies that Eq.(\ref{e}) in Definition 3.1 is satisfied. On the other hand, we have
 \begin{eqnarray*}
D_{k,l}(u)([v,w])&=&-\big[\alpha^{k}\beta^{l}[v,w], u\big]
=\big[u, \alpha^{k+1}\beta^{l-1}[v,w]\big]\\
&=&\big[\beta^{2}(u), [\beta\alpha^{k+1}\beta^{l-2}(v),
\alpha\alpha^{k}\beta^{l-1}(w)]\big]\\
&=&-\big[\alpha^{k+1}\beta^{l}(v),[\alpha^{k}\beta^{l}(w),
\alpha(u)]\big]-\big[\alpha^{k}\beta^{l+1}(w),
[\beta(u),\alpha^{k+2}\beta^{l-2}(v)]\big]\\
&=&-\big[\alpha^{k+1}\beta^{l}(v),[\alpha^{k}\beta^{l}(w),u]\big]
+\big[[u,\alpha^{k+1}\beta^{l-1}(v)],\alpha^{k+1}\beta^{l}(w)\big]\\
&=&[\alpha^{k+1}\beta^{l}(v),D_{k,l}(u)(w) ]+[D_{k,l}(u)(v), \alpha^{k+1}\beta^{l}(w)].
\end{eqnarray*}
Therefore, $D_{k,l}(u)$ is an $\alpha^{k+1}\beta^{l}$-derivation.
Denote  by $Inn_{\alpha^{k}\beta^{l}}(L)$ the set of
inner $\alpha^{k}\beta^{l}$-derivations, i.e.
$$Inn_{\alpha^{k}\beta^{l}}(L)=\{-[\alpha^{k-1}
\beta^{l}(\cdot), u]\mid u\in L, \alpha(u)=u, \beta(u)=u\}.$$

~~~~~~For any $D\in Der_{\alpha^{k}\beta^{l}}(L)$ and
$D^{\prime}\in Der_{\alpha^{s}\beta^{t}}(L)$, 
define their commutator $[D, D^{\prime} ]$  as usual:
\begin{gather}
\label{f}
[D, D^{\prime}]=D\circ D^{\prime}-D^{\prime}\circ D.
\end{gather}
\begin{lemm}
For any $D\in Der_{\alpha^{k}\beta^{l}}(L)$ and $D^{\prime}\in Der_{\alpha^{s}\beta^{t}}(L)$, we have $$[D, D^{\prime}]\in Der_{\alpha^{k+s}\beta^{l+t}}(L).$$
\end{lemm}
\begin{proof}
For any $u,v\in L$, we have
\begin{eqnarray*}
[D, D^{\prime}]([u,v])&=&D\circ D^{\prime}([u,v])-D^{\prime}\circ D([u,v])\\
&=& D([D^{\prime}(u),\alpha^{s}\beta^{t}(v)]+[\alpha^{s}\beta^{t}(u),D^{\prime}(v)])\\
& &- D^{\prime}([D(u),\alpha^{k}\beta^{l}(v)]+[\alpha^{s}\beta^{t}(u), D(v)])\\
&=&[D\circ D^{\prime}(u),\alpha^{k}\beta^{l}\alpha^{s}\beta^{t}(v)]
+[\alpha^{k}\beta^{l}D^{\prime}(u),D\big(\alpha^{s}\beta^{t}(v)\big)] \\
& & +[D\big(\alpha^{s}\beta^{t}(u)\big),\alpha^{k}\beta^{l}D^{\prime}(v)]
+[\alpha^{k}\beta^{l}\alpha^{s}\beta^{t}(u),D\circ D^{\prime}(v)]            \\
& &-[D^{\prime}\circ D(u), \alpha^{s}\beta^{t}\alpha^{k}\beta^{l}(v)]-[\alpha^{s}\beta^{t}D(u), D^{\prime}\big(\alpha^{k}\beta^{l}(v)\big)]\\
& &-[D^{\prime}\big(\alpha^{k}\beta^{l}(u)\big),\alpha^{s}\beta^{t}D(v)]
-[\alpha^{s}\beta^{t}\alpha^{k}\beta^{l}(u),D^{\prime}\circ D(v)]
\end{eqnarray*}
Since any two of maps $D, D^{\prime}, \alpha, \beta$ commute, we have
$$\beta^{l}\circ D^{\prime}=D^{\prime}\circ\beta^{l},
\alpha^{k}\circ D^{\prime}=D^{\prime}\circ\alpha^{k},
\beta^{t}\circ D=D\circ\beta^{t}, \alpha^{s}\circ D=D\circ\alpha^{s}$$
Therefore,we have $$[D, D^{\prime}]([u,v])=\big[[D, D^{\prime}](u),\alpha^{k+s}\beta^{l+t}(v)\big]
+\big[\alpha^{k+s}\beta^{l+t}(u),[D, D^{\prime}](v)\big].$$

Furthermore, it is straightforward to see that
$$[D, D^{\prime}]\circ\alpha=D\circ D^{\prime}\circ\alpha
-D^{\prime}\circ D\circ\alpha=\alpha\circ D\circ D^{\prime}
-\alpha \circ D^{\prime}\circ D=\alpha\circ[D, D^{\prime}],$$
$$[D, D^{\prime}]\circ\beta=D\circ D^{\prime}\circ\beta
-D^{\prime}\circ D\circ\beta=\beta\circ D\circ D^{\prime}-\beta
\circ D^{\prime}\circ D=\beta\circ[D, D^{\prime}],$$
which yields that $[D, D^{\prime}]\in Der_{\alpha^{k+s}\beta^{l+t}}(L).$
\end{proof}

~~~~~~For any integer $k, l$, denote by $Der(L)=\bigoplus_{k,l}
Der_{\alpha^{k}\beta^{l}}(L)$. Obviously, $Der(L)$ is
a Lie algebra, in which the Lie bracket is given by Eq.(\ref{f}).

~~~~~~In the end, we consider the derivation extension of the
regular Bihom-Lie algebra $( L,[-] ,\alpha ,\beta ) $ and give an
application of the $\alpha^{0}\beta^{1}$-derivation $Der_{\alpha^{0}\beta^{1}}(L)$.

~~~~~~ For any linear maps $D, \alpha,\beta: L\rightarrow L$, 
where $\alpha$ and $\beta$ are inverse, 
consider the vector space $L\oplus\R D$. Define a skew-symmetric
bilinear bracket operation $[\cdot,\cdot]_{D}$ on  $L\oplus\R D$ by
 $$[u,v]_{D}=[u,v], \quad [D,u]_{D}=-[\alpha^{-1}\beta(u),
 \alpha\beta^{-1}D]_{D}=D(u), \quad \forall u,v \in L.$$
 Define two linear maps $\alpha_{D},\beta_{D}: L\oplus\R D
 \rightarrow L\oplus\R D$ by $$\alpha_{D}(u, D)=\big(\alpha(u), D\big),
 \quad \beta_{D}(u, D)=\big(\beta(u), D\big).$$
 And the linear maps $\alpha,\beta$ involved in the definition
 of the bracket operation $[\cdot,\cdot]_{D}$ are required to be multiplicative, that is
 $$\alpha\circ[D,u]_{D=}=[\alpha\circ D, \alpha(u)]_{D},
 \quad \beta\circ[D,u]_{D}=[\beta\circ D, \beta(u)]_{D}.$$
 Then, we have
 \begin{eqnarray*}
[u, D]_{D}&=&-[\alpha^{-1}\beta D, \alpha\beta^{-1}(u)]_{D}\\
&=&-\alpha^{-1}\beta\circ[D, \alpha^{2}\beta^{-2}(u)]_{D}\\
&=&-\alpha^{-1}\beta\circ D\big(\alpha^{2}\beta^{-2}(u)\big) \\
&=&-\alpha\beta^{-1}\circ D(u).
\end{eqnarray*}

\begin{theo}
With the above notations, $(L\oplus\R D, [\cdot,\cdot]_{D},\alpha_{D},
\beta_{D})$ is a Bihom-Lie algebra if and only if $D$ is
an $\alpha^{0}\beta^{1}$-derivation of the regular
Bihom-Lie algebra $(L, [\cdot,\cdot], \alpha, \beta).$

\end{theo}

\begin{proof}
For any $u,v\in L, m,n\in\R,$ we have
$$\alpha_{D}\circ\beta_{D}(u, mD)=\alpha_{D}\circ\big(\beta(u), mD\big)
=\big(\alpha\circ\beta(u), mD\big),$$ and
$$\beta_{D}\circ\alpha_{D}(u, mD)=\beta_{D}\circ
\big(\alpha(u), mD\big)=\big(\beta\circ\alpha(u), mD\big).$$
Hence, we have $$\alpha\circ \beta=\beta\circ\alpha \Longleftrightarrow \alpha_{D}\circ\beta_{D}=\beta_{D}\circ\alpha_{D}.$$
On the other hand,
\begin{eqnarray*}
\alpha_{D}[(u, mD),(v,nD)]_{D}&=&\alpha_{D}([u,v]_{D}+[u,nD]_{D}+[mD,v]_{D})\\
&=&\alpha_{D}\big([u,v]-nD\circ\alpha\beta^{-1}(u)+mD(v)\big)\\
&=&\alpha([u,v])-n\alpha\circ D\circ\alpha\beta^{-1}(u)+m\alpha \circ D(v),
\end{eqnarray*}
\begin{eqnarray*}
[\alpha_{D}(u,mD),\alpha_{D}(v,nD)]_{D}&=&[(\alpha(u),mD),(\alpha(v),nD)]_{D}\\
&=&[\alpha(u),\alpha(v)]_{D}+[\alpha(u),nD]_{D}+[mD, \alpha(v)]_{D}\\
&=&[\alpha(u),\alpha(v)]-\alpha\beta^{-1}\circ nD\circ\alpha(u)+mD\circ\alpha(v).
\end{eqnarray*}
Since $\alpha([u,v])=[\alpha(u),\alpha(v)]$, thus
$$\alpha_{D}[(u,mD),(v,nD)]_{D}=[\alpha_{D}(u,mD),\alpha_{D}(v,nD)]_{D}$$
if and only if $$D\circ\alpha=\alpha\circ D, D\circ\beta=\beta\circ D.$$
 Similarly,
 $$\beta_{D}[(u,mD),(v,nD)]_{D}=[\beta_{D}(u,mD),\beta_{D}(v,nD)]_{D}$$
 if and only if $$D\circ\alpha=\alpha\circ D, D\circ\beta=\beta\circ D.$$
Next, we have
\begin{eqnarray*}
[\beta_{D}(u,mD),\alpha_{D}(v,nD)]_{D}&=&[(\beta(u),mD),(\alpha(v),nD)]_{D}\\
&=&[\beta(u),\alpha(v)]_{D}+[\beta(u),nD]_{D}+[mD, \alpha(v)]_{D}\\
&=&-[\beta(v), \alpha(u) ]-\alpha\beta^{-1}\circ nD\circ\beta(u)+mD\circ\alpha(v),
\end{eqnarray*}
\begin{eqnarray*}
[\beta_{D}(v,nD),\alpha_{D}(u,mD)]_{D}&=&[(\beta(v),nD),(\alpha(u),mD)]_{D}\\
&=&[\beta(v),\alpha(u)]_{D}+[\beta(v),mD]_{D}+[nD, \alpha(u)]_{D}\\
&=&[\beta(v),\alpha(u)]-\alpha\beta^{-1}\circ mD\circ\beta(v)+nD\circ\alpha(u),
\end{eqnarray*}
 thus
$$[\beta_{D}(u,mD),\alpha_{D}(v,nD)]_{D}=-[\beta_{D}(v,nD),\alpha_{D}(u,mD)]_{D}$$ if and only if $$D\circ\alpha=\alpha\circ D, D\circ\beta=\beta\circ D.$$
Furthermore,
\begin{eqnarray*}
& &\big[\beta_{D}^{2}(u,mD),[\beta_{D}(v,nD),\alpha_{D}(w,lD)]_{D}\big]_{D}
+\big[\beta_{D}^{2}(v,nD),[\beta_{D}(w,lD),\alpha_{D}(w,mD)]_{D}\big]_{D}\\
& & +\big[\beta_{D}^{2}(w,lD),[\beta_{D}(u,mD),\alpha_{D}(v,nD)]_{D}\big]_{D}\\
&=&\big[(\beta^{2}(u),mD),[(\beta(v),nD), (\alpha(w), lD)]_{D}\big]_{D}\\
& &+\big[(\beta^{2}(v),nD),[(\beta(w),lD), (\alpha(u), mD)]_{D}\big]_{D}\\
& &+\big[(\beta^{2}(w),lD),[(\beta(u),mD), (\alpha(v), nD)]_{D}\big]_{D}\\
&=&\big[(\beta^{2}(u),mD),[\beta(v),\alpha(w)]-l\alpha\circ D(v)+nD\circ\alpha(w)\big]_{D} \\
& & +\big[(\beta^{2}(v),nD),[\beta(w),\alpha(u)]-m\alpha\circ D(w)+lD\circ\alpha(u)\big]_{D}\\
& &+\big[(\beta^{2}(w),lD),[\beta(u),\alpha(v)]-n\alpha\circ D(u)+mD\circ\alpha(v)\big]_{D}\\
&=&\big[\beta^{2}(u),[\beta(v),\alpha(w)]\big]
-[\beta^{2}(u),l\alpha\circ D(v)]+[\beta^{2}(u),nD\circ\alpha(w)]\\
& &+\big[mD,[\beta(v),\alpha(w)]\big]_{D}-[mD,l\alpha\circ D(v)]_{D}+[mD,nD\alpha(w)]_{D}\\
& &+\big[\beta^{2}(v),[\beta(w),\alpha(u)]\big]
-[\beta^{2}(v),m\alpha\circ D(w)]+[\beta^{2}(v),lD\circ\alpha(u)]\\
& &+\big[nD,[\beta(w),\alpha(u)]\big]_{D}-[nD,m\alpha\circ D(w)]_{D}+[nD,lD\alpha(u)]_{D}\\
& &+\big[\beta^{2}(w),[\beta(u),\alpha(v)]\big]
-[\beta^{2}(w),n\alpha\circ D(u)]+[\beta^{2}(w),mD\circ\alpha(v)]\\
& &+\big[lD,[\beta(u),\alpha(v)]\big]_{D}-[lD,n\alpha\circ D(u)]_{D}+[lD,mD\alpha(v)]_{D}\\
&=&\big[\beta^{2}(u),[\beta(v),\alpha(w)]\big]
-[\beta^{2}(u),l\alpha\circ D(v)]+[\beta^{2}(u),nD\circ\alpha(w)]\\
& &+\big[mD,[\beta(v),\alpha(w)]\big]_{D}-ml\alpha\circ D^{2}(v)+mnD^{2}\circ\alpha(w)\\
& &+\big[\beta^{2}(v),[\beta(w),\alpha(u)]\big]
-[\beta^{2}(v),m\alpha\circ D(w)]+[\beta^{2}(v),lD\circ\alpha(u)]\\
& &+\big[nD,[\beta(w),\alpha(u)]\big]_{D}-mn\alpha\circ D^{2}(w)+nlD^{2}\circ\alpha(u)\\
& &+\big[\beta^{2}(w),[\beta(u),\alpha(v)]\big]
-[\beta^{2}(w),n\alpha\circ D(u)]+[\beta^{2}(w),mD\circ\alpha(v)]\\
& &+\big[lD,[\beta(u),\alpha(v)]\big]_{D}
-ln\alpha\circ D^{2}(u)+lmD^{2}\circ\alpha(v),
\end{eqnarray*}
If $D$ is an $\alpha^{0}\beta^{1}$-derivation of the
regular Bihom-Lie algebra $(L, [\cdot,\cdot], \alpha, \beta),$ then
\begin{eqnarray*}
\big[mD,[\beta(v),\alpha(w)]\big]_{D}&=&mD[\beta(v),\alpha(w)]\\
&=&[mD\circ\beta(v),\alpha^{0}\beta\alpha(w)]+[\alpha^{0}\beta^{2}(v),mD\circ\alpha(w)]\\
&=&-[\alpha^{0}\beta^{2}(w),mD\circ\alpha(v)]+[\alpha^{0}\beta^{2}(v),mD\circ\alpha(w)]\\
&=&-[\beta^{2}(w),mD\circ\alpha(v)]+[\beta^{2}(v),mD\circ\alpha(w)].
\end{eqnarray*}
Similarly, $$\big[nD,[\beta(w),\alpha(u)]\big]_{D}
=-[\beta^{2}(u),nD\circ\alpha(w)]+[\beta^{2}(w),n\alpha\circ D(u)],$$
$$\big[lD,[\beta(u),\alpha(v)]\big]_{D}=-[\beta^{2}(v),
lD\circ\alpha(u)]+[\beta^{2}(u),l\alpha\circ D(v)].$$
Therefore, the Bihom-Jacobi identity is satisfied if and only if
$D$ is an $\alpha^{0}\beta^{1}$-derivation of  $(L, [\cdot,\cdot], \alpha, \beta).$
Thus  $(L\oplus\R D, [\cdot,\cdot]_{D},\alpha_{D}, \beta_{D})$
is a Bihom-Lie algebra if and only if $D$ is an
$\alpha^{0}\beta^{1}$-derivation of the regular
Bihom-Lie algebra $(L, [\cdot,\cdot], \alpha, \beta).$
\end{proof}

\section{Representations of Bihom-Lie algebras }

Lie algebra cohomology was introduced by Chevalley and Eilenberg \cite{CE}.
For Hom-Lie algebra, the cohomology theory has been given by \cite{MS2, S, Y2}.
In this section we study representations of Bihom-Lie algebras
and give the corresponding coboundary operators. We can also
construct the semidirect product of Bihom-Lie algebras.
\begin{defi}
Let $(L, [\cdot,\cdot], \alpha, \beta)$ be a Bihom-Lie algebra.
A representation of $L$ is a $4$-tuple
$(M,\rho, \alpha_{M}, \beta_{M})$, where $M$ is a linear space,
$\alpha_{M},\beta_{M}: M\rightarrow M$ are two commuting linear maps
and $\rho: L\rightarrow End(M)$ is a linear map such that, for all $x, y\in L$, we have
\begin{gather}
\rho\big(\alpha(x)\big)\circ\alpha_{M}=\alpha_{M}\circ\rho(x),
\end{gather}
\begin{gather}
\rho\big(\beta(x)\big)\circ\beta_{M}=\beta_{M}\circ\rho(x),
\end{gather}
\begin{gather}
\rho([\beta(x), y])\circ\beta_{M}=\rho\big(\alpha\beta(x)\big)\circ\rho(y)
-\rho\big(\beta(y)\big)\circ\rho\big(\alpha(x)\big).
\end{gather}
\end{defi}

~~~~~~Let $( L,[-] ,\alpha ,\beta ) $ be a regular Bihom-Lie algebra.
The set of $k$-cochains on $L$ with values in $M$, which we denote
by $C^{k}(L; M)$, is the set of skew-symmetric $k$-linear maps
form $L\times\cdots\times L(k-times) $ to $M$:
$$C^{k}(L; M)\triangleq\{ f: \wedge^{k}L\rightarrow M  ~ is ~ a~ linear~ map\}.$$

~~~~~~A $k$-Bihom cochain on $L$ with values in $M$ is defined
to be a $k$-cochain $f\in C^{k}(L; M)$ such that it is
compatible with $\alpha, \beta$ and $\alpha_{M}, \beta_{M}$
in the sense that $\alpha_{M} \circ f=f\circ\alpha$,
$\beta_{M}\circ f=f\circ \beta,$ i.e.
$$\alpha_{M}\big(f(u_{1},\cdots,u_{k})\big)
=f\big(\alpha(u_{1}),\cdots,\alpha(u_{k})\big),$$
$$\beta_{M}\big(f(u_{1},\cdots,u_{k})\big)
=f\big(\beta(u_{1}),\cdots,\beta(u_{k})\big).$$
Denote by $C^{k}_{\substack{(\alpha,\alpha_{M})
\\( \beta, \beta_{M})}}(L; M)\triangleq\{f\in C^{k}(L; M)
\mid \alpha_{M}\circ f=f\circ\alpha, \beta_{M}\circ f=f\circ\beta\}.$

~~~~~~Define $d_{\rho}: C^{k}_{\substack{(\alpha,\alpha_{M})
\\( \beta, \beta_{M})}}(L; M)\rightarrow C^{k+1}(L; M)$ by setting
\begin{eqnarray*}
d_{\rho}f(u_{1},\cdots,u_{k+1})&=&\sum_{i=1}^{k+1}(-1)^{i}
\rho\big(\alpha\beta^{k-1}(u_{i})\big)\big(f(u_{1}, \cdots,
\widehat{u_{i}},\cdots,u_{k+1})\big)\\
&   &+\sum_{i<j}(-1)^{i+j+1}f\big([\alpha^{-1}\beta(u_{i}), u_{j}],
\beta(u_{1}),\cdots,\widehat{u_{i}},\cdots, \widehat{u_{j}}, \cdots, \beta(u_{k+1})\big).
\end{eqnarray*}

\begin{lemm}
With the above notations, for any $f\in C^{k}_{\substack{(\alpha,
\alpha_{M})\\( \beta, \beta_{M})}}(L; M)$, we have
$$(d_{\rho}f)\circ \alpha=\alpha_{M}\circ(d_{\rho}f),$$
$$ (d_{\rho}f)\circ \beta=\beta_{M}\circ(d_{\rho}f).$$
Thus we obtain a well-defined map
$$d_{\rho}: C^{k}_{\substack{(\alpha,\alpha_{M})\\( \beta, \beta_{M})}}(L; M)\rightarrow C^{k+1}_{\substack{(\alpha,\alpha_{M})\\( \beta, \beta_{M})}}(L; M).$$
\end{lemm}
\begin{proof}
For any $f\in C^{k}_{\substack{(\alpha,\alpha_{M})
\\( \beta, \beta_{M})}}(L; M)$, we have  $\alpha_{M} \circ f
=f\circ\alpha$, $\beta_{M}\circ f=f\circ \beta.$
More precisely, we have
\begin{eqnarray*}
& &d_{\rho}f\big(\alpha(u_{1}),\cdots,\alpha(u_{k+1})\big)\\
&=&\sum_{i=1}^{k+1}(-1)^{i}\rho\big(\alpha^{2}\beta^{k-1}(u_{i})
\big)\Big(f\big(\alpha(u_{1}),\cdots,\widehat{u_{i}},
\cdots,\alpha(u_{k+1})\big)\Big)\\
& &+\sum_{i<j}(-1)^{i+j+1}f\big([\alpha^{-1}\beta\alpha(u_{i}), \alpha(u_{j})], \beta\alpha(u_{1}),\cdots,\widehat{u_{i}},
\cdots,\widehat{u_{j}},\cdots,\beta\alpha(u_{k+1})\big)\\
&=&\sum_{i=1}^{k+1}(-1)^{i}\rho\big(\alpha^{2}
\beta^{k-1}(u_{i})\big)\circ\alpha_{M}\circ f(u_{1},
\cdots,\widehat{u_{i}},\cdots,u_{k+1})\\
& &+\sum_{i<j}(-1)^{i+j+1}f\big([\beta(u_{i}), \alpha(u_{j})],
\beta\alpha(u_{1}),\cdots,\widehat{u_{i}},\cdots,
\widehat{u_{j}},\cdots,\beta\alpha(u_{k+1})\big)\\
&=&\sum_{i=1}^{k+1}(-1)^{i}\alpha_{M}\circ\rho\big(\alpha\beta^{k-1}
(u_{i})\big)\big(f(u_{1},\cdots,\widehat{u_{i}},\cdots,u_{k+1})\big)\\
& &+\sum_{i<j}(-1)^{i+j+1}\alpha_{M}\circ f\big([\alpha^{-1}\beta(u_{i}), u_{j}], \beta(u_{1}),\cdots,\widehat{u_{i}},\cdots,
\widehat{u_{j}},\cdots,\beta(u_{k+1})\big)\\
&=&\alpha_{M}\circ d_{\rho}f(u_{1},\cdots,u_{k+1}).
\end{eqnarray*}
Similarly,
\begin{eqnarray*}
& &d_{\rho}f\big(\beta(u_{1}),\cdots,\beta(u_{k+1})\big)\\
&=&\sum_{i=1}^{k+1}(-1)^{i}\rho\big(\alpha\beta^{k}(u_{i})\big)
\Big(f\big(\beta(u_{1}),\cdots,\widehat{u_{i}},\cdots,\beta(u_{k+1})\big)\Big)\\
& &+\sum_{i<j}(-1)^{i+j+1}f\big([\alpha^{-1}\beta\beta(u_{i}), \beta(u_{j})], \beta^{2}(u_{1}),\cdots,\widehat{u_{i}},\cdots,\widehat{u_{j}},
\cdots,\beta^{2}(u_{k+1})\big)\\
&=&\sum_{i=1}^{k+1}(-1)^{i}\rho\big(\alpha\beta^{k}(u_{i})\big)\circ\beta_{M}\circ f(u_{1},\cdots,\widehat{u_{i}},\cdots,u_{k+1})\\
& &+\sum_{i<j}(-1)^{i+j+1}f\big([\alpha^{-1}\beta^{2}(u_{i}),\beta(u_{j})], \beta^{2}(u_{1}),\cdots,\widehat{u_{i}},\cdots,
\widehat{u_{j}},\cdots,\beta^{2}(u_{k+1})\big)\\
&=&\sum_{i=1}^{k+1}(-1)^{i}\beta_{M}\circ\rho\big(\alpha
\beta^{k-1}(u_{i})\big)\big(f(u_{1},\cdots,\widehat{u_{i}},\cdots,u_{k+1})\big)\\
& &+\sum_{i<j}(-1)^{i+j+1}\beta_{M}\circ f\big([\alpha^{-1}\beta(u_{i}), u_{j}], \beta(u_{1}),\cdots,\widehat{u_{i}},\cdots,
\widehat{u_{j}},\cdots,\beta(u_{k+1})\big)\\
&=&\beta_{M}\circ d_{\rho}f(u_{1},\cdots,u_{k+1}).
\end{eqnarray*}
The proof is completed.
\end{proof}
\begin{prop}
The map $d_{\rho}$ is a coboundary operator, i.e. $d_{\rho}^{2}=0.$
\end{prop}
\begin{proof}
For any $f\in C^{k}_{\substack{(\alpha,\alpha_{M})
\\( \beta, \beta_{M})}}(L; M),$  we have
$d_{\rho}^{2}f\in C^{k+2}_{\substack{(\alpha,\alpha_{M})
\\( \beta, \beta_{M})}}(L; M).$
By straightforward computations, we have
\begin{eqnarray*}
& &d_{\rho}^{2}f(u_{1},\cdots,u_{k+2})\\
&=&\sum_{i=1}^{k+2}(-1)^{i}\rho\big(\alpha\beta^{k}(u_{i})\big)
\big(d_{\rho}f(u_{1},\cdots,\widehat{u_{i}},\cdots,u_{k+2})\big)\\
& &+\sum_{i<j}(-1)^{i+j+1}d_{\rho}f\big([\alpha^{-1}\beta(u_{i}), u_{j}],
\beta(u_{1}),\cdots,\widehat{u_{i}},\cdots,
\widehat{u_{j}},\cdots,\beta(u_{k+2})\big).
\end{eqnarray*}
It is not hard to deduce that
\begin{eqnarray}
& &\rho\big(\alpha\beta^{k}(u_{i})\big)\big(d_{\rho}f(u_{1},
\cdots,\widehat{u_{i}},\cdots,u_{k+2})\big)\nonumber\\
&=&\sum_{p=1}^{i-1}(-1)^{p}\rho\big(\alpha\beta^{k}(u_{i})\big)
\circ\rho\big(\alpha\beta^{k-1}(u_{p})\big)(u_{1},\cdots,
\widehat{u_{p}},\cdots,\widehat{u_{i}},\cdots,u_{k+2})\label{aa}\\
& &+\sum_{p=i+1}^{k+2}(-1)^{p-1}\rho\big(\alpha\beta^{k}(u_{i})\big)
\circ\rho\big(\alpha\beta^{k-1}(u_{p})\big)\nonumber\\
& &\quad\times(u_{1},\cdots,\widehat{u_{i}},\cdots,\widehat{u_{p}},\cdots,u_{k+2})\label{bb}\\
& &+\sum_{i<p<q}(-1)^{p+q+1}\rho\big(\alpha\beta^{k}(u_{i})\big)\nonumber\\
& &\quad\times\Big(f\big([\alpha^{-1}\beta(u_{p}),u_{q}],\beta(u_{1}),\cdots,
\widehat{u_{i}},\cdots,\widehat{u_{p}},\cdots,\widehat{u_{q}},
\cdots,\beta(u_{k+2})\big)\Big)\nonumber\\
& &+\sum_{p<q<i}(-1)^{p+q+1}\rho\big(\alpha\beta^{k}(u_{i})\big)\nonumber\\
& &\quad\times\Big(f\big([\alpha^{-1}\beta(u_{p}),u_{q}],\beta(u_{1}),\cdots,
\widehat{u_{p}},\cdots,\widehat{u_{q}},\cdots,\widehat{u_{i}},
\cdots,\beta(u_{k+2})\big)\Big)\nonumber\\
& &+\sum_{p<i<q}(-1)^{p+q}\rho\big(\alpha\beta^{k}(u_{i})\big)\nonumber\\
& &\quad\times\Big(f\big([\alpha^{-1}\beta(u_{p}),u_{q}],\beta(u_{1}),
\cdots,\widehat{u_{p}},\cdots,\widehat{u_{i}},\cdots,\widehat{u_{q}},
\cdots,\beta(u_{k+2})\big)\Big),\nonumber
\end{eqnarray}
and
\begin{eqnarray}
& &d_{\rho}f\big([\alpha^{-1}\beta(u_{i}), u_{j}], \beta(u_{1}),
\cdots,\widehat{u_{i}},\cdots,\widehat{u_{j}},\cdots,
\beta(u_{k+2})\big)\nonumber\\
&=&-\rho\big(\alpha\beta^{k-1}[\alpha^{-1}\beta(u_{i}),u_{j}]\big)
\Big(f\big(\beta(u_{1}),\cdots,\widehat{u_{i}},\cdots,\widehat{u_{j}},
\cdots,\beta(u_{k+2})\big)\Big)\label{ff}\\
& &+\sum_{p=1}^{i-1}(-1)^{p+1}\rho\big(\alpha\beta^{k}(u_{p})\big)
\Big(f\big([\alpha^{-1}\beta(u_{i}),u_{j}],\beta(u_{1}),\cdots,
\widehat{u_{p}},\cdots,\widehat{u_{i}},\cdots,\widehat{u_{j}},
\cdots,\beta(u_{k+2})\big)\Big)\nonumber\\
& &+\sum_{p=i+1}^{j-1}(-1)^{p}\rho\big(\alpha\beta^{k}(u_{p})\big)
\Big(f\big([\alpha^{-1}\beta(u_{i}),u_{j}],\beta(u_{1}),\cdots,
\widehat{u_{i}},\cdots,\widehat{u_{p}},\cdots,\widehat{u_{j}},
\cdots,\beta(u_{k+2})\big)\Big)\nonumber\\
& &+\sum_{j+1}^{k+2}(-1)^{p-1}\rho\big(\alpha\beta^{k}(u_{p})\big)
\Big(f\big([\alpha^{-1}\beta(u_{i}),u_{j}],\beta(u_{1}),\cdots,
\widehat{u_{i}},\cdots,\widehat{u_{j}},\cdots,\widehat{u_{p}},
\cdots,\beta(u_{k+2})\big)\Big)\nonumber\\
& &+\sum_{p=1}^{i-1}(-1)^{p+1}f\big(\big[[\alpha^{-2}\beta^{2}(u_{i}),
\alpha^{-1}\beta(u_{j})],\beta(u_{p})\big],\beta^{2}(u_{1}),\cdots,
\widehat{u_{p}},\cdots,\widehat{u_{i}},\cdots,\widehat{u_{j}},\cdots,
\beta^{2}(u_{k+2})\big)\nonumber\\
\label{j}\\
& &+\sum_{p=i+1}^{j-1}(-1)^{p}f\big(\big[[\alpha^{-2}\beta^{2}(u_{i}),
\alpha^{-1}\beta(u_{j})],\beta(u_{p})\big],\beta^{2}(u_{1}),\cdots,
\widehat{u_{i}},\cdots,\widehat{u_{p}},\cdots,\widehat{u_{j}},
\cdots,\beta^{2}(u_{k+2})\big)\nonumber\\
\label{k}\\
& &+\sum_{p=j+1}^{k+2}(-1)^{p-1}f\big(\big[[\alpha^{-2}\beta^{2}(u_{i}),
\alpha^{-1}\beta(u_{j})],\beta(u_{p})\big],\beta^{2}(u_{1}),\cdots,
\widehat{u_{i}},\cdots,\widehat{u_{j}},\cdots,\widehat{u_{p}},\cdots,
\beta^{2}(u_{k+2})\big)\nonumber\\
\label{l}\\
& &+\sum_{p,q}(\pm)f([\alpha^{-1}\beta\beta(u_{p}),\beta(u_{q})], \beta[\alpha^{-1}\beta(u_{i}),u_{j}],\beta^{2}(u_{1}),\cdots,
\widehat{u_{i,j,p,q}},\cdots,\beta^{2}(u_{k+2})).\label{m}
\end{eqnarray}
In (\ref{m}) , $\widehat{u_{i,j,p,q}}$ means that we omit the
items $u_{i}, u_{j}, u_{p}, u_{q}.$ By the fact that
$[\alpha^{-1}\beta\beta(u_{p}),\beta(u_{q})]
=\beta[\alpha^{-1}\beta(u_{p}),u_{q}],$ we get
$$\sum_{i<j}(-1)^{i+j+1}(\ref{m})=0.$$
By the Bihom-Jacobi identity,
\begin{eqnarray*}
& &\big[ \beta ^{2}(a) , [ \beta  ( a^{\prime } ), \alpha  ( a^{\prime \prime } )  ] \big] \\
&=&\big[ \beta ^{2}(a) , \alpha[ \alpha^{-1}\beta( a^{\prime} ),a^{\prime \prime }  ] \big]\\
&=&-\big[ \beta  [\alpha^{-1}\beta( a^{\prime} ),a^{\prime \prime }],\alpha\beta(a)\big] \\
&=&-\big[[\alpha^{-1}\beta^{2}( a^{\prime} ),\beta(a^{\prime \prime })],\alpha\beta(a)\big] \\
&=&\big[[\alpha^{-1}\beta^{2}(a^{\prime \prime } ),\beta(a^{\prime })],\alpha\beta(a)\big],
\end{eqnarray*}
thus
$$\big[ \beta ^{2}(a) , [ \beta  ( a^{\prime } )
,\alpha  ( a^{\prime \prime } )  ] \big] +\big[ \beta
^{2} ( a^{\prime } ) , [ \beta  ( a^{\prime \prime } )
,\alpha (a)  ] \big] +\big[ \beta ^{2} ( a^{\prime
\prime } ) , [ \beta (a) ,\alpha  ( a^{\prime
} ) ] \big] =0$$   is  equivalent to
$$\big[[\alpha^{-2}\beta^{2}(a^{\prime \prime }),\alpha^{-1}\beta(a^{\prime })], \beta(a)\big]+\big[[\alpha^{-2}\beta^{2}(a^{\prime }),
\alpha^{-1}\beta(a)], \beta(a^{\prime \prime })\big]
+\big[[\alpha^{-2}\beta^{2}(a),\alpha^{-1}
\beta(a^{\prime \prime })], \beta(a^{\prime })\big]=0.$$
We get that $$\sum_{i<j}(-1)^{i+j+1}\big((\ref{j})
+(\ref{k})+(\ref{l})\big)=0.$$
At last, we have that
\begin{eqnarray*}
(\ref{ff})&=&-\rho\big([\beta^{k}(u_{i}),\alpha\beta^{k-1}(u_{j})]\big)
\Big(f\big(\beta(u_{1}),\cdots,\widehat{u_{i}},\cdots,
\widehat{u_{j}},\cdots,\beta(u_{k+2})\big)\Big)\\
&=& -\rho\big([\beta^{k}(u_{i}),\alpha\beta^{k-1}(u_{j})]\big)\circ\beta_{M}\circ f(u_{1},\cdots,\widehat{u_{i}},\cdots,\widehat{u_{j}},\cdots,u_{k+2})\\
&=&-\big[\rho\big(\alpha\beta^{k}(u_{i})\big)\circ\rho
\big(\alpha\beta^{k-1}(u_{j})\big)-\rho\big(\alpha
\beta^{k}(u_{j})\big)\circ\rho\big(\alpha\beta^{k-1}(u_{i})\big)\big]\\
& &\times f(u_{1},\cdots,\widehat{u_{i}},\cdots,\widehat{u_{j}},\cdots,u_{k+2}).
\end{eqnarray*}
Thus we have$$ \sum_{i=1}^{k+2}(-1)^{i}\big((\ref{aa})
+(\ref{bb})\big)+\sum_{i<j}(-1)^{i+j+1}(\ref{ff})=0.$$
The sum of the other six items is zero. Therefore, we
have $d_{\rho}^{2}=0.$ The proof is completed.
\end{proof}

~~~~~~Associated to the representation $\rho$, we obtain
the complex $\big(C^{k}_{\substack{(\alpha,\alpha_{M})
\\( \beta, \beta_{M})}}(L; M), d_{\rho}\big).$ Denote
the set of closed $k$-Bihom-cochains by $Z^{k}_{\alpha,\beta}(L;\rho)$
and the set of exact $k$-Bihom-cochains by $B^{k}_{\alpha,\beta}(L;\rho)$.
Denote the corresponding cohomology by
$$H^{k}_{\alpha,\beta}(L;\rho)=Z^{k}_{\alpha,\beta}
(L;\rho)/B^{k}_{\alpha,\beta}(L;\rho),$$
where
$$
Z^{k}_{\alpha,\beta}(L;\rho)=\{f\in C^{k}_{\substack{(\alpha,\alpha_{M})
\\( \beta, \beta_{M})}}(L; M):d_{\rho}f=0 \},
$$
$$
B^{k}_{\alpha,\beta}(L;\rho)=\{d_{\rho}f^{\prime}:f^{\prime}
\in C^{k-1}_{\substack{(\alpha,\alpha_{M})\\( \beta, \beta_{M})}}(L; M) \}.
$$

~~~~~~In the case of Lie algebras, we can form semidirect products
when given representations. Similarly, we have
\begin{prop}
Let $(L, [\cdot,\cdot], \alpha, \beta)$ be a Bihom-Lie algebra and
$(M, \rho,\alpha_{M},\beta_{M})$ a representation of $L$.
Assume that the maps $\alpha$ and $\beta_{M}$ are bijective.
Then $L\ltimes M:=(L\oplus M, [\cdot,\cdot], \alpha\oplus\alpha_{M},
\beta\oplus\beta_{M})$ is a Bihom-Lie algebra , where
$\alpha\oplus\alpha_{M},\beta\oplus\beta_{M}: L\oplus M\rightarrow L\oplus M$ are defined by $(\alpha\oplus\alpha_{M})(x, a)=\big(\alpha(x),
\alpha_{M}(a)\big)$ and $(\beta\oplus\beta_{M})(x, a)
=\big(\beta(x),\beta_{M}(a)\big),$ and, for all $x, y\in L$
and  $a,b\in M$, the bracket $[\cdot,\cdot]$ is
defined by
$$
[(x,a),(y,b)]=\Big([x,y],
\quad\rho(x)(b)-\rho\big(\alpha^{-1}\beta(y)\big)
\big(\alpha_{M}\beta_{M}^{-1}(a)\big)\Big).
$$
\end{prop}

~~~~~~The proof is straightforward from the definition of the Bihom-Lie
algebra, we are no longer a detailed proof.

\section{The Trivial Representations of Bihom-Lie Algebras}
In this section, we study the trivial representation of regular
Bihom-Lie algebras. Now let $M=\R$, then we have $End(M)=\R$.
Any $\alpha_{M},\beta_{M}\in End(M)$ is just a real number,
which we denote by $r_{1},r_{2}$ respectively. Let $\rho:
L\rightarrow End(M)=\R$ be the zero map. Obviously, $\rho$
is a representation of the regular Bihom-Lie algebra
$(L, [\cdot,\cdot], \alpha, \beta)$ with respect to
any $r_{1},r_{2}\in \R.$ We will always assume that
$r_{1}=r_{2}=1$.  We call this representation
the trivial representation of the regular
Bihom-Lie algebra $(L, [\cdot,\cdot], \alpha, \beta).$

~~~~~~Associated to the trivial representation,
the set of $k$-cochains on $L$, which we denote by $C^{k}(L)$,
is the set of skew-symmetric $k$-linear maps from $L\times\cdots\times L$ (k-times) to $\R$, namely,  $$C^{k}(L)\triangleq\{f: \wedge^{k}L\rightarrow\R
~is ~a ~linear~map\}. $$ i.e.$C^{k}(L)=\wedge^{k}L^{\ast}.$
The set of $k$-Bihom-cochains is given by $$C_{\alpha,\beta}^{k}(L)
=\{f\in \wedge^{k}L^{\ast}\mid f\circ\alpha=f, f\circ\beta=f\}.$$
The corresponding coboundary operator $d_{T}:C_{\alpha,
\beta}^{k}(L)\rightarrow C_{\alpha,\beta}^{k+1}(L)$  is given by
\begin{eqnarray*}
& &d_{T}f(u_{1},\cdots,u_{k+1})\\
&=&\sum_{i<j}(-1)^{i+j+1}f\big( [\alpha^{-1}\beta(u_{i}),u_{j}],\beta(u_{1}),\cdots,\widehat{u_{i}},
\cdots,\widehat{u_{j}},\cdots,\beta(u_{k+1})\big).
\end{eqnarray*}
Denote by $Z^{k}_{\alpha,\beta}(L)$ and $B^{k}_{\alpha,\beta}(L)$
the corresponding closed $k$-Bihom-cochains and exact $k$-Bihom-cochains
respectively. Denote the resulting cohomology by $H^{k}_{\alpha,\beta}(L)$ . Namely,
$$Z^{k}_{\alpha,\beta}(L)=\{f\in C_{\alpha,\beta}^{k}(L): d_{T}f=0 \},$$
$$B^{k}_{\alpha,\beta}(L)=\{d_{T}f^{\prime}: f^{\prime}\in C_{\alpha,\beta}^{k-1}(L) \},$$
$$H^{k}_{\alpha,\beta}(L)=Z^{k}_{\alpha,\beta}(L)/B^{k}_{\alpha,\beta}(L).$$
\begin{prop}
With the above notations, associated to the trivial
representation of the regular Bihom-Lie algebra $(L, [\cdot,\cdot], \alpha, \beta)$ , we have
\begin{eqnarray*}
& &H^{0}_{\alpha,\beta}(L)=\R,\\
& &H^{1}_{\alpha,\beta}(L)=\{f\in C_{\alpha,\beta}^{1}(L)\mid
f\big([\beta(u),\alpha( v)]\big)=0, \forall u,v\in L\}.
\end{eqnarray*}
\end{prop}
\begin{proof}
Obviously, any $s\in \R$ is a 0-Bihom-cochain. By the definition of
coboundary operator $d_{T}$, we have $d_{T}s=0$.
Thus we have $H^{0}_{\alpha,\beta}(L)=\R.$
For any $f\in C_{\alpha,\beta}^{1}(L)$,
we have
$$d_{T}f(u,v)=f\big([\alpha^{-1}\beta(u), v]\big)=f\circ\alpha\big([\alpha^{-1}\beta(u), v]\big)=f\big([\beta(u),\alpha(v)]\big).$$
Therefore, $f$ is closed if and only if
$f\big([\beta(u),\alpha(v)]\big)=0 ,~ \forall u,v\in L$.
The conclusion follows from the fact that there
is no exact 1-Bihom-cochain. i.e.$B^{1}_{\alpha,\beta}(L)=0.$
 Thus we have $H^{1}_{\alpha,\beta}(L)=\{f\in C_{\alpha,
 \beta}^{1}(L)\mid f\big([\beta(u),\alpha( v)]\big)=0, \forall u,v\in L\}$.
\end{proof}

~~~~~~In the following we consider central extensions
of the multiplicative  Bihom-Lie algebra $(L, [\cdot,\cdot],
\alpha, \beta).$  Obviously, $(\R, 0, 1, 1)$
is a multiplicative Bihom-Lie algebra with the trivial
bracket and the identity morphism. Let $\theta\in
C_{\alpha,\beta}^{2}(L)$, we have $\theta\circ\alpha
=\theta$ and  $\theta\circ\beta=\theta$. We consider
the direct sum $\mathfrak{g}=L\oplus\R$ with the following bracket
$$[(u,s),(v,t)]_{\theta}=\big([u,v],\theta(\alpha\beta^{-1}(u),v)\big),
\quad \forall u,v\in L, s,t\in \R.$$
Define $\alpha_{\mathfrak{g}},\beta_{\mathfrak{g}}:
\mathfrak{g}\rightarrow\mathfrak{g}$ by
$\alpha_{\mathfrak{g}}(u,s)=\big(\alpha(u), s\big),
\beta_{\mathfrak{g}}(u,s)=\big(\beta(u), s\big).$
\begin{theo}
With the above notations, $(\mathfrak{g}, [\cdot,\cdot]_{\theta},
\alpha_{\mathfrak{g}},\beta_{\mathfrak{g}})$
is a multiplicative Bihom-Lie algebra if and only if
$\theta\in C_{\alpha,\beta}^{2}(L)$ is a
$2$-cocycle associated to the trivial representation, i.e.
$$d_{T}\theta=0.$$
\end{theo}
\begin{proof}
Obviously, since $\alpha\circ\beta=\beta\circ\alpha,$
we have $\alpha_{\mathfrak{g}}\circ\beta_{\mathfrak{g}}
=\beta_{\mathfrak{g}}\circ\alpha_{\mathfrak{g}}.$
Then we show that $\alpha_{\mathfrak{g}}$
is an algebra morphism with the respect to
the bracket $[\cdot,\cdot]_{\theta}$. On one hand, we have
$$\alpha_{\mathfrak{g}}[(u,s),(v,t)]_{\theta}=
\alpha_{\mathfrak{g}}\big([u,v],\theta(\alpha\beta^{-1}(u),v)\big)
=\big(\alpha([u,v]),\theta(\alpha\beta^{-1}(u),v)\big).$$
On the other hand, we have
$$[\alpha_{\mathfrak{g}}(u,s),\alpha_{\mathfrak{g}}(v,t)]_{\theta}
=\big[\big(\alpha(u),s\big),\big(\alpha(v),t\big)\big]_{\theta}
=\Big([\alpha(u),\alpha(v)], \theta
\big(\alpha\beta^{-1}\alpha(u),\alpha(v)\big)\Big).$$
Since $\alpha$ is an algebra morphism and  $\theta
\big(\alpha\beta^{-1}\alpha(u),\alpha(v)\big)=\theta(\alpha\beta^{-1}(u),v),$
we deduce that $\alpha_{\mathfrak{g}}$ is an algebra morphism.
Similarly, we have $\beta_{\mathfrak{g}}$ is also an algebra morphism.
Furthermore,  we have
\begin{eqnarray*}
[\beta_{\mathfrak{g}}(u,s),\alpha_{\mathfrak{g}}(v,t)]_{\theta}
&=&[\big(\beta(u),s\big),\big(\alpha(v),t\big)]_{\theta}\\
&=&\Big([\beta(u),\alpha(v)],\theta
\big(\alpha\beta^{-1}\beta(u),\alpha(v)\big)\Big)\\
&=&\Big([\beta(u),\alpha(v)],\theta
\big(\alpha(u),\alpha(v)\big)\Big)\\
&=&\Big([\beta(u),\alpha(v)],\theta
(u,v)\Big)
\end{eqnarray*}
and
\begin{eqnarray*}
[\beta_{\mathfrak{g}}(v,t),\alpha_{\mathfrak{g}}(u,s)]_{\theta}
&=&[\big(\beta(v),t\big),\big(\alpha(u),s\big)]_{\theta}\\
&=&\Big([\beta(v),\alpha(u)],\theta
\big(\alpha\beta^{-1}\beta(v),\alpha(u)\big)\Big)\\
&=&\Big([\beta(v),\alpha(u)],\theta
\big(\alpha(v),\alpha(u)\big)\Big)\\
&=&\Big([\beta(v),\alpha(u)],\theta
(v,u)\Big).
\end{eqnarray*}
By (\ref{b}), since $\theta$ is a skew-symmetric $2$-linear map,  we obtain that
$$[\beta_{\mathfrak{g}}(u,s),\alpha_{\mathfrak{g}}(v,t)]_{\theta}
=-[\beta_{\mathfrak{g}}(v,t),\alpha_{\mathfrak{g}}(u,s)]_{\theta}.$$
By direct computations, on the hand,we have
\begin{eqnarray*}
& &\big[\beta_{\mathfrak{g}}^{2}(u,s),[\beta_{\mathfrak{g}}(v,t),
\alpha_{\mathfrak{g}}(w,m)]_{\theta}\big]_{\theta}
+c.p.\big((u,s), (v,t), (w,m)\big)\\
&=&\big[\big(\beta^{2}(u),s\big),[\big(\beta(v),t\big),
\big(\alpha(w),m\big)]_{\theta}\big]_{\theta}+c.p.\big((u,s), (v,t), (w,m)\big)\\
&=&\Big[\big(\beta^{2}(u),s\big),\Big([\beta(v),\alpha(w)],
\theta\big(\alpha(v),\alpha(w)\big)\Big)\Big]_{\theta}
+c.p.\big((u,s), (v,t), (w,m)\big)\\
&=&\big[ \beta ^{2}(u),[\beta( v ),\alpha( w )] \big]+c.p.(u,v,w)+\theta\big(\alpha\beta(u),[\beta(v),\alpha(w)]\big)+c.p.(u,v,w).
\end{eqnarray*}
Thus by the Bihom-Jacobi identity of $L$, $[\cdot,\cdot]_{\theta}$
satisfies the Bihom-Jacobi identity if and only if
$$\theta\big(\alpha\beta(u),[\beta(v),\alpha(w)]\big)
+\theta\big(\alpha\beta(v),[\beta(w),\alpha(u)]\big)
+\theta\big(\alpha\beta(w),[\beta(u),\alpha(v)]\big)=0.$$
Namely,$$ \theta\big(\beta(u),[\alpha^{-1}\beta(v),w]\big)
+\theta\big(\beta(v),[\alpha^{-1}\beta(w),u]\big)
+\theta\big(\beta(w),[\alpha^{-1}\beta(u),v]\big)=0.$$
On the other hand,
\begin{eqnarray*}
& &d_{T}\theta(u,v,w)\\
&=&\theta\big([\alpha^{-1}\beta(u),v],\beta(w)\big)
-\theta\big([\alpha^{-1}\beta(u),w],\beta(v)\big)
+\theta\big([\alpha^{-1}\beta(v),w],\beta(u)\big)\\
&=&-\theta\big(\beta(w),[\alpha^{-1}\beta(u),v]\big)
-\theta\big(\beta(v),[\alpha^{-1}\beta(w),u]\big)
-\theta\big(\beta(u),[\alpha^{-1}\beta(v),w]\big)
\end{eqnarray*}
Therefore, $(\mathfrak{g}, [\cdot,\cdot]_{\theta},
\alpha_{\mathfrak{g}},\beta_{\mathfrak{g}})$
is a multiplicative Bihom-Lie algebra if and only
if $\theta\in C_{\alpha,\beta}^{2}(L)$ is a $2$-cocycle
associated to the trivial representation, i.e.
$d_{T}\theta=0.$
\end{proof}

~~~~~~We call the multiplicative Bihom-Lie algebra
$(\mathfrak{g}, [\cdot,\cdot]_{\theta},
\alpha_{\mathfrak{g}},\beta_{\mathfrak{g}})$ the
\textit{central extension }of $(L, [\cdot,\cdot],
\alpha, \beta)$ by the abelian Bihom-Lie algebra $(\R, 0, 1, 1).$

\begin{prop}
For $\theta_{1},\theta_{2}\in Z^{2}(L),$ if $\theta_{1}
-\theta_{2}$ is a exact, the corresponding two central
extensions $(\mathfrak{g}, [\cdot,\cdot]_{\theta_{1}},
\alpha_{\mathfrak{g}},\beta_{\mathfrak{g}})$
and $(\mathfrak{g}, [\cdot,\cdot]_{\theta_{2}},
\alpha_{\mathfrak{g}},\beta_{\mathfrak{g}})$ are isomorphic.
\end{prop}
\begin{proof}
Assume that $\theta_{1}-\theta_{2}=d_{T}f$,
$f\in C_{\alpha,\beta}^{1}(L).$ Thus we have
\begin{eqnarray*}
& &\theta_{1}\big(\alpha\beta^{-1}(u),v\big)
-\theta_{2}\big(\alpha\beta^{-1}(u),v\big)\\
&=&d_{T}f\big(\alpha\beta^{-1}(u),v\big)\\
&=&f\big([\alpha^{-1}\beta\circ\alpha\beta^{-1}(u),v]\big)\\
&=&f([u,v]).
\end{eqnarray*}
Define $\varphi_{\mathfrak{g}}:\mathfrak{g}\rightarrow\mathfrak{g}$ by
$$\varphi_{\mathfrak{g}}(u,s)=\big(u,s-f(u)\big).$$
Obviously, $\varphi_{\mathfrak{g}}$ is an isomorphism of vector spaces.
The fact that $\varphi_{\mathfrak{g}}$
is a morphism of the Bihom-Lie algebra follows
from the fact $\theta\circ\alpha=\theta,
\theta\circ\beta=\theta.$ More precisely, we have
$$
\varphi_{\mathfrak{g}}\circ\alpha_{\mathfrak{g}}(u,s)
=\varphi_{\mathfrak{g}}\big(\alpha(u),s\big)
=\big(\alpha(u),s-f\circ\alpha(u)\big)=\big(\alpha(u),s-f(u)\big).
$$
On the other hand, we have
$$\alpha_{\mathfrak{g}}\circ\varphi_{\mathfrak{g}}(u,s)
=\alpha_{\mathfrak{g}}\big(u,s-f(u)\big)
=\big(\alpha(u),s-f(u)\big).
$$
Thus, we obtain that $\varphi_{\mathfrak{g}}\circ\alpha_{\mathfrak{g}}
=\alpha_{\mathfrak{g}}\circ\varphi_{\mathfrak{g}}.$
Similarly, we have $\varphi_{\mathfrak{g}}\circ\beta_{\mathfrak{g}}
=\beta_{\mathfrak{g}}\circ\varphi_{\mathfrak{g}}.$
We also have
\begin{eqnarray*}
\varphi_{\mathfrak{g}}[(u,s),(v,t)]_{\theta_{1}}&=&\varphi_{\mathfrak{g}}
\Big([u,v],\theta_{1}\big(\alpha\beta^{-1}(u),v\big)\Big)\\
&=&\Big([u,v],\theta_{1}\big(\alpha\beta^{-1}(u),v\big)-f([u,v])\Big)\\
&=&\Big([u,v],\theta_{2}\big(\alpha\beta^{-1}(u),v\big)\Big)\\
&=&[\varphi_{\mathfrak{g}}(u,s),\varphi_{\mathfrak{g}}(v,t)]_{\theta_{2}}.
\end{eqnarray*}
Therefore, $\varphi_{\mathfrak{g}}$ is also an isomorphism of Bihom-Lie algebras.
\end{proof}

\section{The Adjoint Representations of Bihom-Lie Algebras}

Let $(L, [\cdot,\cdot], \alpha, \beta)$ be a regular Bihom-Lie algebra.
We consider that $L$ represents on itself via the bracket
with respect to the morphisms $\alpha, \beta$. 
\begin{defi}
For any integer $s ,t$, the $\alpha^{s}\beta^{t}$-adjoint
representation of the regular Bihom-Lie algebra
$(L, [\cdot,\cdot], \alpha, \beta)$, which we denote by $ad_{s,t}$, is defined by
$$ad_{s,t}(u)(v)=[\alpha^{s}\beta^{t}(u),v], ~~\forall u,v\in L.$$
\end{defi}
\begin{lemm}
With the above notations, we have
$$ad_{s,t}\big(\alpha(u)\big)\circ\alpha=\alpha\circ ad_{s,t}(u);$$
$$ad_{s,t}\big(\beta(u)\big)\circ\beta=\beta\circ ad_{s,t}(u);$$
$$ad_{s,t}\big([\beta(u),v]\big)\circ\beta=ad_{s,t}\big(\alpha\beta(u)\big)\circ ad_{s,t}(v)-ad_{s,t}\big(\beta(v)\big)\circ ad_{s,t}\big(\alpha(u)\big).$$
Thus the definition of $\alpha^{s}\beta^{t}$-adjoint representation is well defined.
\end{lemm}
\begin{proof}
For any $u,v,w\in L,$ first we show that $ad_{s,t}\big(\alpha(u)\big)\circ\alpha=\alpha\circ ad_{s,t}(u).$
\begin{eqnarray*}
ad_{s,t}\big(\alpha(u)\big)\circ\alpha(v)&=&[\alpha^{s}\beta^{t}\circ\alpha(u),\alpha(v)]\\
&=&\alpha\big([\alpha^{s}\beta^{t}(u),v]\big)=\alpha\circ ad_{s,t}(u)(v).
\end{eqnarray*}
Similarly, we have $$ad_{s,t}\big(\beta(u)\big)\circ\beta(v)=\beta\circ ad_{s,t}(u)(v).$$
Note that the skew-symmetry condition implies $$ad_{s,t}(u)(v)
=[\alpha^{s}\beta^{t}(u),v]=-[\alpha^{-1}\beta(v),
\alpha^{s+1}\beta^{t-1}(u)], ~~\forall u,v\in L.$$
On one hand,  we have
\begin{eqnarray*}
\big(ad_{s,t}([\beta(u),v])\circ\beta\big)(w)&=&ad_{s,t}([\beta(u),v])\big(\beta(w)\big)\\
&=&-\big[ \alpha^{-1}\beta^{2}(w),\alpha^{s+1}\beta^{t-1}[\beta(u),v]\big]\\
&=&-\big[ \alpha^{-1}\beta^{2}(w),[\alpha^{s+1}\beta^{t}(u),\alpha^{s+1}\beta^{t-1}(v)]\big].
\end{eqnarray*}
On the other hand, we have
\begin{eqnarray*}
& &\big(ad_{s,t} (\alpha\beta(u))\circ ad_{s,t}(v)\big)
(w)-\big(ad_{s,t} (\beta(v))\circ ad_{s,t}(\alpha(u))\big)(w)\\
&=&ad_{s,t} \big(\alpha\beta(u)\big)\big(-[\alpha^{-1}\beta(w),\alpha^{s+1}\beta^{t-1}(v)]\big)\\
& &-ad_{s,t} \big(\beta(v)\big)\big(-[\alpha^{-1}\beta(w),\alpha^{s+2}\beta^{t-1}(u)]\big)\\
&=&\big[\alpha^{-1}\beta[\alpha^{-1}\beta(w), \alpha^{s+1}\beta^{t-1}(v)],\quad\alpha^{s+1}\beta^{t-1}\alpha\beta(u)\big]\\
& &-\big[\alpha^{-1}\beta[\alpha^{-1}\beta(w), \alpha^{s+2}
\beta^{t-1}(u)],\quad\alpha^{s+1}\beta^{t-1}\beta(v)\big]\\
&=&\big[\beta[\alpha^{-2}\beta(w), \alpha^{s}\beta^{t-1}(v)],\quad\alpha^{s+2}\beta^{t}(u)\big]\\
& &-\big[\beta[\alpha^{-2}\beta(w), \alpha^{s+1}\beta^{t-1}(u)],\quad\alpha^{s+1}\beta^{t}(v)\big]\\
&\overset{\text{skew-symmetry}}{=}&-\big[\beta\alpha^{s+1}\beta^{t}(u),\quad\alpha[\alpha^{-2}\beta(w), \alpha^{s}\beta^{t-1}(v)]\big]\\
& &+\big[\beta\alpha^{s}\beta^{t}(v),\quad\alpha[\alpha^{-2}\beta(w), \alpha^{s+1}\beta^{t-1}(u)]\big]\\
&=&-\big[\alpha^{s+1}\beta^{t+1}(u),\quad[\alpha^{-1}\beta(w), \alpha^{s+1}\beta^{t-1}(v)]\big]\\
& &+\big[\alpha^{s}\beta^{t+1}(v),\quad[\alpha^{-1}\beta(w), \alpha^{s+2}\beta^{t-1}(u)]\big]\\
&=&\big[\alpha^{s+1}\beta^{t+1}(u),\quad[\alpha^{s}\beta^{t}(v),  w]\big]\\
& &+\big[\alpha^{s}\beta^{t+1}(v),\quad[\alpha^{-1}\beta(w), \alpha^{s+2}\beta^{t-1}(u)]\big]\\
&=&\big[\beta^{2}\big(\alpha^{s+1}\beta^{t-1}(u)\big),\quad[\beta\big(\alpha^{s}\beta^{t-1}(v)\big), \alpha\big(\alpha^{-1}(w)\big)]\big]\\
& &+\big[\beta^{2}\big(\alpha^{s}\beta^{t-1}(v)\big),\quad[\beta\big(\alpha^{-1}(w)\big),  \alpha\big(\alpha^{s+1}\beta^{t-1}(u)\big)]\big]\\
&=&-\big[\beta^{2}\big(\alpha^{-1}(w)\big),\quad[\beta\big(\alpha^{s+1}\beta^{t-1}(u)\big), \alpha\big(\alpha^{s}\beta^{t-1}(v)\big)]\big]\\
&=&-\big[ \alpha^{-1}\beta^{2}(w),\quad[\alpha^{s+1}\beta^{t}(u), \alpha^{s+1}\beta^{t-1}(v)]\big].
\end{eqnarray*}
Thus, the definition of $\alpha^{s}\beta^{t}$-adjoint
representation is well defined. The proof is completed.
\end{proof}

~~~~~~The set of $k$-Bihom-cochains on $L$ with coefficients
in $L$, which we denote by $C_{\alpha,\beta}^{k}(L; L)$, are given by
$$C_{\alpha,\beta}^{k}(L; L)=\{f\in C^{k}(L; L)
\mid f\circ\alpha=\alpha\circ f, f\circ\beta=\beta\circ f\}.$$
The set of 0-Bihom-cochains are given by:
$$C_{\alpha,\beta}^{0}(L; L)=\{u\in L\mid\alpha(u)=u, \beta(u)=u\}.$$

~~~~~~Associated to the $\alpha^{s}\beta^{t}$-adjoint representation, the coboundary operator $d_{s,t}: C_{\alpha,\beta}^{k}(L; L)\rightarrow C_{\alpha,\beta}^{k+1}(L; L)$ is given by
\begin{eqnarray*}
& &d_{s,t}f(u_{1},\cdots,u_{k+1})\\
&=&\sum_{i=1}^{k+1}(-1)^{i}[\alpha^{s+1}\beta^{t+k-1}(u_{i}),
f(u_{1},\cdots,\widehat{u_{i}},\cdots,u_{k+1})]\\
&   &+\sum_{i<j}(-1)^{i+j+1}f\big([\alpha^{-1}\beta(u_{i}), u_{j}], \beta(u_{1}),\cdots,\widehat{u_{i}},\cdots,\widehat{u_{j}},\cdots,\beta(u_{k+1})\big).
\end{eqnarray*}

 ~~~~~~For the $\alpha^{s}\beta^{t}$-adjoint representation $ad_{s,t}$,
 we obtain the $\alpha^{s}\beta^{t}$-adjoint complex
 $(C_{\alpha,\beta}^{k}(L;L), d_{s,t})$ and the corresponding cohomology $$H^{k}(L;ad_{s,t})=Z^{k}(L;ad_{s,t})/B^{k}(L;ad_{s,t}).$$

~~~~~~We have known that a $1$-cocycle associated to the
adjoint representation is a derivation for Lie algebras and Hom-Lie algebras. Similarly, we have
\begin{prop}\label{6.3}
Associated to the  $\alpha^{s}\beta^{t}$-adjoint
representation $ad_{s,t}$ of the regular Bihom-Lie
algebra $(L, [\cdot,\cdot], \alpha, \beta)$,
$D\in C_{\alpha,\beta}^{1}(L;L)$ is a $1$-cocycle
if and only if $D$ is an $\alpha^{s+2}\beta^{t-1}$-derivation,
i.e. $D\in Der_{\alpha^{s+2}\beta^{t-1}}(L).$
\end{prop}
\begin{proof}
The conclusion follows directly from the definition of
the coboundary operator $d_{s,t}$. $D$ is closed if and only if
$$d_{s,t}(D)(u,v)=-[\alpha^{s+1}\beta^{t}(u),D(v)]
+[\alpha^{s+1}\beta^{t}(v),D(u)]+D([\alpha^{-1}\beta(u),v])=0.$$
$D$ is an $\alpha^{s+2}\beta^{t-1}$-derivation if and only if
\begin{eqnarray*}
D([\alpha^{-1}\beta(u),v])&=&[\alpha^{s+2}\beta^{t-1}
\alpha^{-1}\beta(u),D(v)]+[\alpha^{-1}\beta \circ D(u),\alpha^{s+2}\beta^{t-1}(v)]\\
&=&[\alpha^{s+1}\beta^{t}(u),D(v)]-[\alpha^{s+1}\beta^{t}(v),D(u)].
\end{eqnarray*}
So $D\in C_{\alpha,\beta}^{1}(L;L)$ is a $1$-cocycle
if and only if $D$ is an $\alpha^{s+2}\beta^{t-1}$-derivation.
\end{proof}
\begin{prop}
For any integer $s, t$, Associated to the
$\alpha^{s}\beta^{t}$-adjoint representation $ad_{s,t}$, we have
$$H^{0}(L;ad_{s,t})=\{u\in L\mid\alpha(u)=u, \beta(u)=u, [u,v]=0, ~~\forall v\in L\};$$
$$H^{1}(L;ad_{s,t})=Der_{\alpha^{s+2}\beta^{t-1}}(L)/ Inn_{\alpha^{s+2}\beta^{t-1}}(L).$$
\end{prop}
\begin{proof}
For any 0-Bihom-cochain $u\in C_{\alpha,\beta}^{0}(L;L)$ , we have
$$d_{s,t}u(v)=-[\alpha^{s+1}\beta^{t-1}(v),u]=[\alpha^{-1}
\beta(u),\alpha^{s+2}\beta^{t-2}(v)],~~\forall v\in L.$$
Therefore, $u$ is a closed $0$-Bihom-cochain if and only if
$$[\alpha^{-1}\beta(u),\alpha^{s+2}\beta^{t-2}(v)]=0,$$
which is equivalent to
$$\alpha^{-s-2}\beta^{-t+2}([\alpha^{-1}\beta(u),\alpha^{s+2}\beta^{t-2}(v)])=[u,v]=0.$$
Therefore, the set of closed $0$-Bihom-cochain $Z^{0}(L;ad_{s,t})$ is given by
$$Z^{0}(L;ad_{s,t})=\{u\in C_{\alpha,\beta}^{0}(L;L)\mid [u,v]=0, ~~\forall v\in L\}.$$
Obviously,
$$B^{0}(L;ad_{s,t})=0.$$
Thus we have $$H^{0}(L;ad_{s,t})=\{u\in L\mid\alpha(u)=u,
\beta(u)=u, [u,v]=0, ~~\forall v\in L\}.$$
By Proposition \ref{6.3}, we have $Z^{1}(L;ad_{s,t})
=Der_{\alpha^{s+2}\beta^{t-1}}(L).$ Furthermore,
it is obvious that any exact $1$-Bihom-cochain is of
the form $-[\alpha^{s+1}\beta^{t-1}(\cdot), u]$
for some $u\in C_{\alpha,\beta}^{0}(L;L).$
Therefore, we have $B^{1}(L;ad_{s,t})
= Inn_{\alpha^{s+2}\beta^{t-1}}(L),$
which implies that $H^{1}(L;ad_{s,t})
=Der_{\alpha^{s+2}\beta^{t-1}}(L)/
Inn_{\alpha^{s+2}\beta^{t-1}}(L).$
\end{proof}
\vs{5pt}

\noindent{\bbb{References}}

\end{document}